\documentclass[12pt]{amsart}
\usepackage{amsmath, amsthm, amssymb, amsopn, amsfonts, enumerate}
\usepackage{stmaryrd}
\newtheorem{thm}{Theorem}
\newtheorem{prop}[thm]{Proposition}

\newtheorem{conj}{Conjecture}

\newcommand{\Q}{\mathbb{Q}}

\newcommand{\Z}{\mathbb{Z}}
\renewcommand{\P}{\mathbb{P}}
\newcommand{\M}{\mathcal{M}}

\newcommand{\Mbar}{\overline{\M}}

\DeclareMathOperator{\DR}{DR}
\DeclareMathOperator{\Aut}{Aut}

\title{Generalized boundary strata classes}
\author{A. Pixton}
\begin{document}

\maketitle

\begin{abstract}
We describe a generalization of the usual boundary strata classes in the Chow ring of $\Mbar_{g,n}$. The generalized boundary strata classes additively span a subring of the tautological ring. We describe a multiplication law satisfied by these classes and check that every double ramification cycle lies in this subring.
\end{abstract}

\setcounter{section}{-1}
\section{Introduction}
Let $g,n\ge 0$ satisfy $2g-2+n > 0$. These notes will be concerned with certain classes in the Chow ring of the moduli space $\Mbar_{g,n}$ of stable curves of genus $g$ with $n$ marked points. All of our classes will be {\it tautological}; that is, they are contained in a certain subring $R^*(\Mbar_{g,n}) \subseteq A^*(\Mbar_{g,n})$, the tautological ring. The full structure of the tautological ring is not well understood, though a large family of relations is known \cite{PPZ}.

The double ramification (DR) cycle is an important class in $R^*(\Mbar_{g,n})$, defined in terms of the virtual class in relative Gromov-Witten theory. Informally, the DR cycle parametrizes curves admitting a map to $\P^1$ with specified ramification profiles over two points. Faber and Pandharipande \cite{FP:relative} proved that the DR cycle lies in the tautological ring, and then an explicit formula for the DR cycle in terms of basic tautological classes was recently proven in \cite{JPPZ}. Trying to simplify this formula was one of the main motivations for these notes.

We will define the {\it boundary subring} $B^*(\Mbar_{g,n})$ of the tautological ring $R^*(\Mbar_{g,n})$. The boundary subring is defined as the additive span of certain generalized boundary strata classes $[\Gamma]$ corresponding to topological types of prestable curves. The boundary subring has a convenient multiplication law and may be of interest in its own right. The DR cycle is contained in the boundary subring and admits a simpler formula in terms of these generalized boundary strata classes.

In Section~\ref{sec:tautological}, we recall some of the basic definitions and properties of the tautological ring and fix some notation for dual graphs. In Section~\ref{sec:boundary}, we define the generalized boundary strata classes and describe the multiplication law in the boundary subring. In Section~\ref{sec:gfree}, we define the genus-free boundary subring, a natural subring of the boundary subring.
Finally, in Section~\ref{sec:DR} we explain how to write the double ramification cycle in terms of generalized boundary strata classes.

\subsection{Acknowledgements}
I would like to thank G. Oberdieck for useful discussions when first trying to define generalized boundary strata classes. I would also like to thank an anonymous referee.

I was supported by a fellowship from the Clay Mathematics Institute.

\section{The tautological ring}\label{sec:tautological}
Following Faber and Pandharipande \cite{Faber-Pandharipande}, the tautological rings $R^*(\Mbar_{g,n})$ are defined simultaneously for all $g,n\ge 0$ satisfying $2g-2+n>0$ as the smallest subrings of the Chow ring $A^*(\Mbar_{g,n})$ closed under pushforward by forgetful maps $\Mbar_{g,n+1}\to\Mbar_{g,n}$ and gluing maps $\Mbar_{g,n+2}\to\Mbar_{g+1,n}$ or $\Mbar_{g_1,n_1+1}\times\Mbar_{g_2,n_2+1}\to\Mbar_{g_1+g_2,n_1+n_2}$. In this section we will recall a more explicit description of the tautological ring given by Graber and Pandharipande \cite[Appendix A]{Graber-Pandharipande}. They described a set of additive generators and a multiplication law satisfied by these generators.

The generators can be indexed by connected finite graphs decorated with certain additional structures. In this paper we will be interested in several slightly different types of decorated graphs, so we go through these structures one at a time. We begin by letting a {\it graph with $n$ legs} mean a connected finite graph $\Gamma$ along with a sequence of vertices (possibly with repetition) $l_1,l_2,\ldots,l_n\in V(\Gamma)$ of length $n\ge 0$. Here $l_i$ is the location of the $i$th leg of $\Gamma$. We think of legs as half-edges, so the valence $n_v$ of a vertex $v$ includes a contribution of one for each $l_i$ equal to $v$.

Then a {\it graph of genus $g$ with $n$ legs} is a graph $\Gamma$ with $n$ legs along with a nonnegative integer $g_v$ for each vertex $v\in V(\Gamma)$, satisfying the equality
\[
g = h_1(\Gamma) + \sum_{v\in V(\Gamma)}g_v,
\]
where $h_1(\Gamma) = |E(\Gamma)| - |V(\Gamma)| + 1$ is the cycle number of $\Gamma$.

Next, a graph $\Gamma$ of genus $g$ with $n$ legs is {\it stable} if it satisfies the condition
\[
2g_v-2+n_v > 0
\]
at each vertex $v\in V(\Gamma)$. We will usually just call such a graph a {\it stable graph} and treat the values of $g$ and $n$ as understood.

Stable graphs correspond to strata in the boundary stratification of $\Mbar_{g,n}$ as follows. First, given a stable graph $\Gamma$ we write
\[
\Mbar_{\Gamma} := \prod_{v\in V(\Gamma)}\Mbar_{g_v,n_v}.
\]
The edges of $\Gamma$ then define a gluing map
\[
\xi_\Gamma:\Mbar_\Gamma\to \Mbar_{g,n}.
\]
The image of this map is the closed boundary stratum corresponding to $\Gamma$, the closure of the locus of curves with dual graph $\Gamma$. Up to a factor of $|\Aut(\Gamma)|$, the class of this stratum is equal to the pushforward ${\xi_\Gamma}_*1$.

The Chow ring of each factor $\Mbar_{g_v,n_v}$ of $\Mbar_\Gamma$ contains basic tautological classes: the Arbarello-Cornalba \cite{Arbarello-Cornalba} kappa classes $\kappa_j$ ($j>0$) of codimension $j$ and the cotangent line classes $\psi_i$ ($1\le i \le n_v$) of codimension $1$. We can then consider classes
\[
{\xi_\Gamma}_*\alpha \in A^*(\Mbar_{g,n})
\]
given any stable graph $\Gamma$ and any monomial $\alpha$ in the kappa and psi classes on the factors of $\Mbar_\Gamma$. These are precisely the additive generators for the tautological ring considered in \cite{Graber-Pandharipande}.

In other words, the generators can be indexed by the data of a stable graph $\Gamma$ with additional kappa and psi decorations (on the vertices and half-edges respectively). We will describe this as a stable graph $\Gamma$ with a monomial decoration $\alpha$.

The multiplication rule for these classes given by Graber and Pandharipande \cite[eq. (11)]{Graber-Pandharipande} has a somewhat complicated form. Let $\Gamma_1,\Gamma_2$ be two stable graphs with monomial decorations $\alpha_1,\alpha_2$ respectively. Then
\begin{equation}\label{eq:GP_mult}
({\xi_{\Gamma_1}}_*\alpha_1) \cdot ({\xi_{\Gamma_1}}_*\alpha_2) = \sum_{\Gamma,\widetilde{\alpha}_1,\widetilde{\alpha}_2,E_1,E_2}\frac{1}{\Aut(\Gamma)}\mathbf{c}\cdot{\xi_\Gamma}_*(\widetilde{\alpha}_1\widetilde{\alpha}_2\Psi).
\end{equation}
This formula requires some explanation. First, the sum runs over (representatives of isomorphism classes of) stable graphs $\Gamma$ along with two monomial decorations $\widetilde{\alpha}_1,\widetilde{\alpha}_2$ on $\Gamma$ and two disjoint sets of edges $E_1,E_2\subseteq E(\Gamma)$ such that $\widetilde{\alpha_i}$ does not use any psi classes located on the edges in $E_i$.

The coefficient $\mathbf{c}\in\Z_{\ge 0}$ then counts the number of ways to choose isomorphisms (of monomial-decorated graphs) $(\Gamma/E_i,\widetilde{\alpha}_i)\cong (\Gamma_i,\alpha_i)$ for $i=1,2$, where $\Gamma/E_i$ is the stable graph formed by contracting the edges in $E_i$ (and combining vertex genera in the natural way, including adding one whenever a loop is contracted). The factor $\Psi$ is the product of $(-\psi_1-\psi_2)$ over all edges in $E(\Gamma)\setminus(E_1\cup E_2)$, where $\psi_1,\psi_2$ are the $\psi$ classes on the two sides of the edge.

Note that the multiplication law \eqref{eq:GP_mult} is quite complicated even when the monomials $\alpha_1,\alpha_2$ are both equal to $1$. In particular, the product of two boundary strata classes $({\xi_{\Gamma_1}}_*1)\cdot({\xi_{\Gamma_2}}_*1)$ is not necessarily a sum of boundary strata classes because of the $\Psi$ factor. The generalized boundary strata classes defined in the next section will simultaneously fix this issue and yield a simpler multiplication law.

\section{Generalized boundary strata classes}\label{sec:boundary}
The goal of this section is to define classes
\[
[\Gamma]\in R^d(\Mbar_{g,n})
\]
for any graph $\Gamma$ of genus $g$ with $n$ legs (with $2g-2+n > 0$), where $d = |E(\Gamma)|$ is the number of edges in $\Gamma$. The definition will satisfy $[\Gamma] = {\xi_{\Gamma}}_*1$ for any stable graph $\Gamma$. When $\Gamma$ is unstable, $[\Gamma]$ will provide a convenient generalization of the usual boundary strata classes. For example, if $\Gamma$ is the graph with two vertices connected by a single edge, and one of the vertices is genus $0$ and has leg $i$ and no other legs, then we will have $[\Gamma] = -\psi_i$, the first Chern class of the line bundle on $\Mbar_{g,n}$ given by the tangent space at the $i$th marked point.

We define $[\Gamma]$ in two steps. First we define $[\Gamma]$ for graphs of genus $g$ with $n$ legs satisfying the semistability condition $2g_v-2+n_v \ge 0$ at every vertex $v\in V(\Gamma)$. In other words, the only type of unstable vertex that can occur is a genus $0$ vertex of valence $2$. These unstable vertices are naturally partitioned into paths, and there are two different types of paths of unstable vertices: a path of $k \ge 1$ unstable vertices between two stable vertices $v_1,v_2$ (possibly equal) or a path of $k \ge 1$ unstable vertices between a stable vertex $v$ and a leg labeled $i$.

In the first case, we contract the path to a single edge between $v_1$ and $v_2$ and decorate it with $(-\psi_1-\psi_2)^k/(k+1)!$. In the second case, we contract the path to a single leg on $v$ labeled $i$ and decorate it with $(-\psi)^k/k!$. After performing these actions to every path of unstable vertices, we are left with a stable graph $\Gamma_{\text{cont}}$ decorated by a polynomial $\alpha$ in the psi classes on $\Mbar_{\Gamma_{\text{cont}}}$. We then set
\[
[\Gamma] := {\xi_{\Gamma_{\text{cont}}}}_*\alpha,
\]
so $[\Gamma]$ is a linear combination of the usual tautological ring generators described in Section~\ref{sec:tautological}.

Now we define $[\Gamma]$ for general graphs $\Gamma$ of genus $g$ with $n$ legs. Let $v_1,\ldots,v_m$ be the vertices of $\Gamma$ with genus $0$ and valence $1$ (in some order). Create a new graph $\Gamma'$ by attaching new legs labeled $n+1,\ldots,n+m$ to $v_1,\ldots,v_m$ respectively (i.e. $l_{n+i} = v_i$). Then $2g_v-2+n_v \ge 0$ at every vertex $v\in V(\Gamma')$, so we can use our previous definition to obtain a class $[\Gamma'] \in R^d(\Mbar_{g,n+m})$. Then take
\[
[\Gamma] := (\pi_m)_*((-\psi_{n+1})\cdots(-\psi_{n+m})[\Gamma']) \in R^d(\Mbar_{g,n}),
\]
where $\pi_m:\Mbar_{g,n+m}\to\Mbar_{g,n}$ is the map forgetting the last $m$ markings. (This pushforward will create kappa classes in general: see \cite[eq. (1.12)]{Arbarello-Cornalba}.)

This completes the definition of $[\Gamma]$ for any graph $\Gamma$ of genus $g$ with $n$ legs. We now let
\[
B^*(\Mbar_{g,n})\subseteq R^*(\Mbar_{g,n})
\]
be the additive span of the classes $[\Gamma]$ inside the tautological ring. We have the following multiplication law, a much simpler version of \eqref{eq:GP_mult}:
\begin{prop}\label{prop:mult1}
Let $\Gamma_1,\Gamma_2$ be two graphs of the same genus $g$ and number of legs $n$. Then
\[
[\Gamma_1]\cdot[\Gamma_2] = \sum_{\Gamma}\frac{1}{|\Aut(\Gamma)|}c_{\Gamma_1,\Gamma_2,\Gamma}[\Gamma],
\]
where $c_{\Gamma_1,\Gamma_2,\Gamma}$ is the number of ways of choosing a partition of the edges $E(\Gamma) = E_1 \sqcup E_2$ along with isomorphisms $\Gamma/E_i \cong \Gamma_i$ for $i=1,2$.
\end{prop}
\begin{proof}
This follows from \eqref{eq:GP_mult} along with the definition of $[\Gamma]$. We sketch part of the proof here to give the idea.

Suppose that we have chosen a graph $\Gamma$ along with a choice of the data counted by $c_{\Gamma_1,\Gamma_2,\Gamma}$ (i.e. $E_1$, $E_2$, and isomorphisms $\Gamma/E_i \cong \Gamma_i$). We want to match the contribution $[\Gamma]$ with (part of) the general multiplication law \eqref{eq:GP_mult}.

First, if $\Gamma$ is a stable graph then $\Gamma_1,\Gamma_2$ are also stable (since contractions of a stable graph are stable). Then $[\Gamma_i] = {\xi_{\Gamma_i}}_*1$ and the same data ($E_1$, $E_2$, isomorphisms) describes a term $[\Gamma]$ in \eqref{eq:GP_mult}.

Next, suppose $\Gamma$ is a semistable graph. Then as previously explained, the unstable vertices can be partitioned into paths of two types: those between two stable vertices and those between a stable vertex and a leg. In either case, suppose such a path consists of $m$ edges. Then these edges are partitioned between $E_1$ and $E_2$; suppose that $m_i$ of them are in $E_i$, so $m_1 + m_2 = m$. For the contracted graphs $\Gamma/E_i$ the values of $m_1$ and $m_2$ are all that matter, so the contribution appears with multiplicity $\binom{m}{m_1}$. Then the matching with \eqref{eq:GP_mult} follows from the simple identities
\[
\frac{(-\psi_1-\psi_2)^{m_1-1}}{m_1!}\cdot\frac{(-\psi_1-\psi_2)^{m_2-1}}{m_2!}\cdot(-\psi_1-\psi_2) = \binom{m}{m_1}\cdot\frac{(-\psi_1-\psi_2)^{m-1}}{m!}
\]
and
\[
\frac{(-\psi)^{m_1}}{m_1!}\cdot\frac{(-\psi)^{m_2}}{m_2!} = \binom{m}{m_1}\cdot\frac{(-\psi)^m}{m!}
\]
for the two types of paths. In the first case, the extra factor of $(-\psi_1-\psi_2)$ comes from the factor $\Psi$ in \eqref{eq:GP_mult}.

The case where $\Gamma$ is non-semistable is handled similarly. Here the definition of $[\Gamma]$ produces polynomials in the kappa classes (via the pushforward formula for psi class monomials, see \cite[eq. (1.13)]{Arbarello-Cornalba}), and a slightly more complicated combinatorial identity is needed to get the desired matching with terms in \eqref{eq:GP_mult}; we omit this for brevity.
\end{proof}

We conclude that $B^*(\Mbar_{g,n})$ is a ring. We call it the {\it boundary subring}.

Let $\mathcal{G}_{g,n}$ be a $\Q$-vector space with basis indexed by isomorphism classes of graphs of genus $g$ with $n$ legs. The multiplication law above (Proposition~\ref{prop:mult1}) defines a multiplication operation on $\mathcal{G}_{g,n}$. It is easily checked that this formal graph algebra $\mathcal{G}_{g,n}$ is a commutative, associative graded $\Q$-algebra with unit given by the graph with no edges, and then $[-]:\mathcal{G}_{g,n}\to R^*(\Mbar_{g,n})$ is a homomorphism with image $B^*(\Mbar_{g,n})$.

\vspace{5mm}
\noindent{\bf Remark.}
It is worth noting that the multiplication law becomes even simpler if we rescale our generalized boundary strata class by a factor of $|\Aut(\Gamma)|$, since we can rewrite it as
\[
\frac{[\Gamma_1]}{|\Aut(\Gamma_1)|}\cdot\frac{[\Gamma_2]}{|\Aut(\Gamma_2)|} = \sum_{\Gamma}c'_{\Gamma_1,\Gamma_2,\Gamma}\frac{[\Gamma]}{|\Aut(\Gamma)|},
\]
where $c'_{\Gamma_1,\Gamma_2,\Gamma}$ is the number of ways of choosing a partition of the edges $E(\Gamma) = E_1 \sqcup E_2$ such that $\Gamma/E_i$ is isomorphic to $\Gamma_i$ for $i=1,2$. This also makes it easier to check that the multiplication is associative. However, our chosen normalization makes the statement of Proposition~\ref{prop:forget} below much nicer.
\vspace{5mm}

The most important property of the classes $[\Gamma]$ has already been discussed: they generate a subring $B^*(\Mbar_{g,n})$ with a relatively simple multiplication law. They also satisfy simple identities under pullback and pushforward by forgetful maps:
\begin{prop}\label{prop:forget}
Let $g,n\ge 0$ satisfy $2g-2+n > 0$. Let $\Gamma$ be a graph of genus $g$ with $n$ legs. Then we have:
\begin{enumerate}
\item If $\pi:\Mbar_{g,n+1}\to\Mbar_{g,n}$ is the map forgetting marking $n+1$, then
\[
\pi^*[\Gamma] = \sum_{v\in V(\Gamma)}[\Gamma_v],
\]
where $\Gamma_v$ is formed from $\Gamma$ by attaching leg $n+1$ at vertex $v$.
\item If $n>0$ and $\pi:\Mbar_{g,n}\to\Mbar_{g,n-1}$ is the map forgetting marking $n$, then
\[
\pi_*(\psi_n[\Gamma]) = (2g_v-2+n_v)[\Gamma'],
\]
where $\Gamma'$ is the graph formed from $\Gamma$ by removing the $n$th leg and $g_v,n_v$ are the genus and valence in $\Gamma'$ at the vertex where the leg was removed.
\end{enumerate}
\end{prop}

As with Proposition~\ref{prop:mult1}, this proposition is a straightforward combinatorial consequence of the analogous, more complicated results for general tautological classes.

\section{Genus-free boundary classes}\label{sec:gfree}
We will now define a natural subring of the boundary subring $B^*(\Mbar_{g,n})$. The idea is to sum over all possible distributions of the available genus across the vertices of the graph. Let $G$ be a graph with $n$ legs (but no genus assignments). For any nonnegative integer $g$ satisfying $2g-2+n > 0$, we define
\[
[G]_g := \sum_{\substack{g_v \ge 0\text{ for }v \in V(G) \\ \sum g_v = g - h_1(G)}}[\Gamma] \in R^*(\Mbar_{g,n}),
\]
where the $\Gamma$ inside the sum is the genus-labeled graph given by $G$ and the $g_v$.

We then immediately have from Proposition~\ref{prop:mult1} that these genus-free boundary classes additively span a subring $GB^*(\Mbar_{g,n})\subseteq B^*(\Mbar_{g,n})$, with the following multiplication law:
\begin{prop}\label{prop:mult2}
Let $n\ge 0$ and let $G_1, G_2$ be two graphs with $n$ legs. Then for any $g\ge 0$ satisfying $2g-2+n > 0$, we have
\[
[G_1]_g\cdot[G_2]_g = \sum_{G}\frac{1}{|\Aut(G)|}c_{G_1,G_2,G}[G]_g,
\]
where $c_{G_1,G_2,G}$ is the number of ways of choosing a partition of the edges $E(G) = E_1 \sqcup E_2$ along with isomorphisms $G/E_i \cong G_i$ for $i=1,2$.
\end{prop}
We call this subring $GB^*(\Mbar_{g,n})$ the {\it genus-free boundary subring}. As before, we can define a formal graph algebra $\mathcal{G}_n$ (now without genus information in the graphs) and view $[-]_g$ as a homomorphism from this algebra to $BG^*(\Mbar_{g,n})$.

Since the structure of $\mathcal{G}_n$ is genus-invariant, it is natural to make the following stability conjecture about the genus-free boundary subring:
\begin{conj}
For fixed $d,n\ge 0$, the homomorphism $[-]_g: \mathcal{G}_n \to BG^*(\Mbar_{g,n})$ is an isomorphism in degree $d$ for all sufficiently large $g$.
\end{conj}
For example, for $d = n = 1$ we do not have an isomorphism for $g=1$ since then $[\Gamma]_1 = \kappa_1 - \psi_1 = 0 \in R^*(\Mbar_{1,1})$ for $\Gamma$ the unique connected graph with two vertices, one edge, and one leg. But then for $g\ge 2$ it is easily checked that the map $[-]_g$ is an isomorphism in degree $1$.

\vspace{5mm}
\noindent{\bf Remark.}
We've now defined three subrings of the Chow ring:
\[
BG^*(\Mbar_{g,n})\subseteq B^*(\Mbar_{g,n}) \subseteq R^*(\Mbar_{g,n}).
\]
It is reasonable to ask whether they are actually distinct. The difference between $B^*(\Mbar_{g,n})$ and $R^*(\Mbar_{g,n})$ is that psi classes on opposite sides of an edge must appear together in the boundary subring, so in small genus (e.g. $g\le 3$) we can use tautological relations to remove most psi classes and see that $B^*(\Mbar_{g,n}) = R^*(\Mbar_{g,n})$. For larger values of $g$, say $g\ge 6$, these rings should become different. In contrast, already for $\Mbar_{1,3}$ the subring $BG^*(\Mbar_{g,n})$ is strictly smaller than the other two.

\section{Rewriting the double ramification cycle formula}\label{sec:DR}
Let $a_1,\ldots,a_n$ be integers with sum zero. Taking the positive $a_i$ and the negative $a_i$ separately gives two partitions $\mu$, $\nu$ of the same number $k$. Also, let $n_0$ count the number of $a_i$ equal to zero. The double ramification cycle $\DR_g(a_1,\ldots,a_n) \in R^g(\Mbar_{g,n})$ can then be viewed as parametrizing curves admitting a degree $k$ map to $\P^1$ with ramification profiles $\mu$, $\nu$ over points $0,\infty$ respectively and $n_0$ internal marked points. More precisely, let $\Mbar_{g,n_0}(\P^1/\{0,\infty\},\mu,\nu)^\sim$ be the moduli space of stable relative maps to a rubber $\P^1$ (see \cite[Section 0.2.3]{FP:relative}). Then
\[
\DR_g(a_1,\ldots,a_n) = p_*[\Mbar_{g,n_0}(\P^1/\{0,\infty\},\mu,\nu)^\sim]^{\text{vir}},
\]
the pushforward of the virtual class along the natural map
\[
p: \Mbar_{g,n_0}(\P^1/\{0,\infty\},\mu,\nu)^\sim \to \Mbar_{g,n}.
\]

A complicated but explicit formula for $\DR_g(a_1,\ldots,a_n)$ in terms of tautological classes was recently proven in \cite[Theorem 1]{JPPZ}. In this section we explain how this formula can be simplified by means of the genus-free boundary classes $[G]_g$ from Section~\ref{sec:gfree}.

The DR cycle formula was written in \cite[Section 0.4.2]{JPPZ} as a sum over stable graphs decorated by certain power series in the psi classes. The psi classes appear in these power series either as the psi class associated to a leg of the graph or as the sum of the two psi classes associated to an edge of the graph. These are precisely the two types of psi decorations that appear in generalized boundary strata, and it turns out that we can simply remove the psi power series and enlarge the sum to run over the unstable graphs as well as the stable graphs.

In addition, the coefficients in this sum do not depend on how the genus is distributed over the vertices, so we can pass to the genus-free boundary strata classes. The result is the following formula.

Let $d\ge 0$ and $r > 0$ be integers. Then we can define
\begin{multline}\label{eq:DR}
\DR_g^{d,r}(a_1,\ldots,a_n) =\\
\sum_G \left(\frac{1}{r^{h_1(G)}}\sum_{w:H(G)\to \{0,\ldots,r-1\}}\prod_{e = (h,h') \in E(G)}\frac{1}{2}w(h)w(h')\right) \frac{[G]_g}{|\Aut(G)|},
\end{multline}
where the first sum is over all graphs with $n$ legs and exactly $d$ edges and the second sum is over all weightings mod $r$ of $G$, as defined in \cite[Section 0.4.1]{JPPZ}.

Then the actual DR cycle $\DR_g(a_1,\ldots,a_n)$ is the constant term in $r$ of $\DR_g^{g,r}(a_1,\ldots,a_n)$ (which is polynomial in $r$ for $r$ sufficiently large as explained in \cite{JPPZ}). In particular, $\DR_g(a_1,\ldots,a_n)$ belongs to the genus-free boundary subring $GB^*(\Mbar_{g,n})$.

\vspace{5mm}
\noindent{\bf Remark.}
There are two natural generalizations of the DR cycle formula that are discussed in \cite{JPPZ}. We will briefly discuss how they fit into our boundary ring framework.

First, the DR cycle formula is really the degree $g$ part of a class of impure degree. The parts of this formula in degrees greater than $g$ give tautological relations, as conjectured in \cite{JPPZ} and proved by Clader and Janda \cite{Clader-Janda}. These relations in degree $d > g$ are then just the constant term in $r$ of the class $\DR_g^{d,r}(a_1,\ldots,a_n)$ in \eqref{eq:DR}. In other words, the $r$-constant term of the formula \eqref{eq:DR} can be interpreted as defining a class $\DR^d(a_1,\ldots,a_n)$ in the formal graph algebra $\mathcal{G}_n$ from Section~\ref{sec:gfree}, and this class has the property that its image in $BG^*(\Mbar_{g,n})$ is equal to the DR cycle for $g = d$ and equal to zero for $g < d$.

Second, there is a $k$-twisted version of the DR cycle formula \cite[Section 1.1]{JPPZ}. This formula yields a class in the boundary subring $B^*(\Mbar_{g,n})$ but no longer in the genus-free boundary subring $GB^*(\Mbar_{g,n})$. Formula \eqref{eq:DR} will still hold with two modifications: $G$ must be replaced by a graph $\Gamma$ of genus $g$ with $n$ legs and $w$ must now be a $k$-weighting mod $r$ (this condition for $k\ne 0$ depends on the genus distribution across the vertices of the graph). The extra kappa classes appearing in the $k$-twisted DR cycle formula are assimilated into our definition of the generalized boundary strata classes just as the psi classes were in the untwisted version.

\bibliographystyle{amsplain}
\bibliography{boundary-arxiv}

\vspace{+8 pt}
\noindent
Department of Mathematics\\
Massachusetts Institute of Technology\\
apixton@mit.edu

\end{document}